\documentclass[a4paper,12pt]{article}
\usepackage{amsmath,amsfonts,enumerate,amssymb,amsthm,color}

\title{A classification of nilpotent $3$-BCI groups}
\author{Hiroki  Koike\footnote{Supported in part by ARRS - Agencija za raziskovanje Republike Slovenija, program no. P1-0285.} \\
{\small  UP IAM and UP FAMNIT, University of Primorska}  \\  [-0.8ex]
{\small  Muzejski trg 2, SI-6000 Koper, Slovenia} \\ 
{\small  \texttt{hiroki.koike@upr.si}}
\and 
 Istv\'an Kov\'acs\footnote{Supported in part by ARRS - Agencija za raziskovanje Republike Slovenija, program no. P1-0285 and project N1-0011 (ESF EUROCORES EUROGiga/GReGAS).} \\
{\small  UP IAM and UP FAMNIT, University of Primorska}  \\  [-0.8ex]
{\small  Muzejski trg 2, SI-6000 Koper, Slovenia} \\ 
{\small  \texttt{istvan.kovacs@upr.si}}
}

\def\C{\mathcal{C}}
\def\Cs{\mathcal{C}_{\mathrm{sub}}}
\def\ee{\mathbf{1}}
\def\G{\mathcal{G}}
\def\H{\mathcal{H}}

\def\nn{\mathbf{0}}
\def\Q{Q}

\def\Z{\mathbb{Z}}
\DeclareMathOperator{\aut}{Aut}
\DeclareMathOperator{\bcay}{BCay}
\DeclareMathOperator{\cay}{Cay}

\DeclareMathOperator{\orb}{Orb}
\DeclareMathOperator{\sym}{Sym}

\newtheorem{thm}{Theorem}[section]
\newtheorem{lem}[thm]{Lemma}
\newtheorem{cor}[thm]{Corollary}
\newtheorem{prop}[thm]{Proposition}

\theoremstyle{remark}
\newtheorem{rem}[thm]{Remark}

\begin{document}

\maketitle

\begin{abstract}
Given a finite group $G$ and a subset $S\subseteq G,$ the bi-Cayley graph $\bcay(G,S)$  is the graph whose vertex 
set is $G \times \{0,1\}$ and edge set is  
$\{ \{(x,0),(s x,1)\}  : x \in G, s\in S \}$. A bi-Cayley graph $\bcay(G,S)$ is called a BCI-graph if for any bi-Cayley graph 
$\bcay(G,T),$ $\bcay(G,S) \cong \bcay(G,T)$ implies that  $T =  g S^\alpha$ for some $g \in G$ and $\alpha \in \aut(G)$. 
A group $G$ is called an $m$-BCI-group if all bi-Cayley graphs of $G$ of valency at most $m$ are BCI-graphs.
In this paper we prove that, a finite nilpotent group is a $3$-BCI-group if and only if it is in the form $U \times V,$
where $U$ is a homocyclic group of odd order, and $V$ is trivial or one of the groups $\Z_{2^r},$ $\Z_2^r$ and $\Q_8$. 

\medskip\noindent{\it Keywords:} bi-Cayley graph, BCI-group, graph  isomorphism.

\medskip\noindent{\it MSC 2010:}  05C25.
\end{abstract}

\section{Introduction}

In this paper every group and every (di)graph will be finite. 
Given a  group $G$ and a subset $S\subseteq G,$ the \emph{bi-Cayley graph} $\bcay(G,S)$ of $G$ 
with respect to $S$ is the graph whose vertex set is $G \times \{0,1\}$ and edge set is  
$\{ \{(x,0),(s x,1)\}  : x \in G, s\in S \}$.  
We call two bi-Cayley graphs $\bcay(G,S)$ and $\bcay(G,T)$  \emph{bi-Cayley isomorphic} if 
$T = g S^\alpha$ for some $g \in G$ and $\alpha \in \aut(\Gamma)$ 
(here and in what follows for $x \in G$ and $R \subseteq G,$   $x R = \{ x r : r  \in R\}$).  
It is an easy exercise to  show that bi-Cayley isomorphic bi-Cayley graphs are 
isomorphic as usual graphs. The converse implication is not true in general, and this makes the following definition
 interesting (see \cite{XuJSL08}):  a bi-Cayley graph $\bcay(G,S)$ is a \emph{BCI-graph} if for any bi-Cayley graph 
$\bcay(G,T)$, $\bcay(G,S) \cong \bcay(G,T)$ implies that $T =  g S^\alpha$ for some  $g \in G$ and $\alpha \in \aut(G)$.  
A group $G$ is called an \emph{$m$-BCI-group} if all bi-Cayley graphs of $G$ of valency at most $m$ are BCI-graphs, 
and an $|G|$-BCI-group is simply called a \emph{BCI-group}.

\begin{rem}
It should be remarked that each of the above concepts has a natural analog in the theory of Cayley digraphs. 
Recall that, for a group $G$ and a subset $S\subseteq G,$ the \emph{Cayley digraph} $\cay(G,S)$ 
is the digraph whose vertex set is $G$ and arc set is $\{  (x,s x)  : x \in G, s\in S \}$. 
A Cayley digraph $\cay(G,S)$ is called a \emph{CI-graph} if for any Cayley digraph 
$\cay(G,T)$, $\cay(G,S) \cong \cay(G,T)$ implies that $T =  S^\alpha$ for some $\alpha \in \aut(G),$ 
$G$ is called an \emph{$m$-DCI-group} if all Cayley digraphs of $G$ of valency at most $m$ are 
CI-graphs, and an $(|G|-1)$-CI-group is simply called a \emph{CI-group}. 
Finite CI-groups and $m$-DCI-groups have  attracted considerable attention over the last $40$ years.
For more information on these groups, the reader is referred to the survey \cite{Li02}.
\end{rem}

The study of $m$-BCI-groups was initiated in \cite{XuJSL08}, where it was shown that every group is a $1$-BCI-group, 
and a group is a $2$-BCI-group if and only if it has the property that 
any two elements of the same order are either fused or inverse fused (these groups are described in \cite{LiP97}). 
The problem  of classifying all $3$-BCI-groups  is still open. Up to our knowledge, it is only known that every cyclic group is a 
$3$-BCI-group (this is a consequence of \cite[Theorem 1.1]{WieZ07}, see also \cite{JinL09}), 
and that $A_5$ is the only non-Abelian  simple $3$-BCI-group (see \cite{JinL10}). 
The purpose of this paper is to make a further step  by classifying the nilpotent $3$-BCI-groups. 

In fact, we have a relatively short list of candidates for nilpotent $3$-BCI groups, which arises from the earlier 
works of W. Jin and W. Liu  \cite{JinL10,JinL11} on the Sylow $p$-subgroups of $3$-BCI-groups. 
In particular, a Sylow $2$-subgroup of a $3$-BCI-group is $\Z_{2^r},$ $\Z_2^4$ or the quaternion group $\Q_8$ 
(see \cite{JinL10}), while a  Sylow $p$-subgroup for $p > 2$ is homocyclic (see \cite{JinL11}). 
A group is said to be \emph{homocyclic} if it is a direct product of cyclic groups of the same order. 
Consequently, if $G$ is  a nilpotent $3$-BCI-group, then $G$ decomposes as 
$G = U \times V,$ where $U$ is a homocyclic group of odd order, and $V$ is trivial or one of the groups $\Z_{2^r},$ $\Z_2^r$ and $\Q_8$. In this paper we prove that the converse implication also holds, and by this complete the classification of 
nilpotent $3$-BCI-groups. Our main result is the following theorem:

\begin{thm}\label{Main}
Every finite group $U \times V$  is a $3$-BCI-group, 
where $U$ is a homocyclic group of odd order, and $V$ is trivial or one of the groups $\Z_{2^r},$ $\Z_2^r$ and $\Q_8$. 
\end{thm}

\begin{rem}
It is interesting to compare the known $3$-BCI-groups with the class of $2$-DCI-groups. 
It follows from the classification  of finite $2$-DCI-groups \cite[Theorem 1.3]{Li99} that, 
$A_5$ is the only non-Abelian simple $2$-DCI-group, and the 
nilpotent $2$-DCI-groups are exactly those given in Theorem \ref{Main}. 
However, a $3$-BCI-group is not always a $2$-DCI-group. A rather exhausted analysis shows that the Frobenius group 
$\Z_3^3 \rtimes \Z_4$ is a $3$-BCI-group, while it is not a $2$-DCI group, which can be seen from \cite[Theorem 1.3]{Li99}.
\end{rem}

We prove Theorem \ref{Main} in two parts. In Section 2 we treat the case when $\bcay(G,S)$ is 
not arc-transitive, $G = U \times V,$ where 
$U$ is a homocyclic group of odd order, and $V$ is trivial or one of $\Z_{2^r},$ $\Z_2^r$ and $\Q_8,$ 
and $|S|=3$. The rest will be done in Section 3.  

\section{Non-arc-transitive BCI-graphs}

We start by fixing the relevant notation and terminology. 
Let $G$ be a finite group acting on a finite set $V$. 
For a subset $U \subseteq V,$ denote by $G_U$ the
elementwise stabilizer of $U$ in $G$, while by $G_{\{U\}}$ the
setwise stabilizer of $U$ in $G$. If $U = \{u_1,\ldots,u_k\},$ then $G_{u_1,\ldots,u_k}$ 
will be written for $G_{\{u_1,\ldots,u_k\}}$, in particular, we write $G_u$ for $G_{\{u\}}$. 
We say that the subset $U$ is \emph{$G$-invariant} if $G$ leaves $U$ setwise fixed, or equivalently, when 
$G_{\{U\}} = G$. Suppose, in addition, that $G$ acts transitively on $V$. 
A subset $\Delta \subseteq V$ is called a \emph{block} for $G$ if for every $g \in G,$ 
$\Delta^g = \Delta$ or $\Delta \cap \Delta^g = \emptyset$. 
The sets $\Delta^g, g \in G,$ form a partition of $V,$ which is called the \emph{system of blocks for $G$ induced by $\Delta$}.
Denoted this partition by $\delta,$ $G$ acts on $\delta$ naturally. The corresponding \emph{kernel}  
will be denoted by $G_\delta,$ i.e., $G_\delta = \{ g \in G :  \Delta'^{\, g}=\Delta' \text{ for all } \Delta' \in \delta \}$.

For a graph $\Gamma$, we let $V(\Gamma)$, $E(\Gamma),$ $A(\Gamma),$ and $\aut(\Gamma)$ denote the 
vertex set, the edge set, the arc set, and the full group of automorphisms of $\Gamma$, respectively.  
For a subset $U \le V(\Gamma),$ we let $\Gamma[U]$ denote the \emph{subgraph of $\Gamma$ induced by $U$}.  
A graph $\Gamma$ is called \emph{arc-transitive} when $\aut(\Gamma)$ is transitive on $A(\Gamma)$. 
We let $K_n$ and $K_{n,n}$ denote the complete graph on $n$ vertices and the complete bipartite 
graph on $2n$ vertices respectively. By a \emph{cubic graph} we simply mean a regular graph of valency $3$.  

Throughout the paper $\C$ denotes the set of all groups $U \times V,$ where $U$ is 
a homocyclic group of odd order, and $V$ is either trivial or one of $\Z_{2^r},$ $\Z_2^r$ and $\Q_8;$ 
and $\Cs$ denotes the set of all groups that have an overgroup in $\C$. \medskip

The main result in this section is the following theorem, which we are going to prove in the end of the section:

\begin{thm}\label{Main-non-AT}
Let $\Gamma=\bcay(G,S),$ $G \in \C,$ $|S|=3,$ and suppose that $\Gamma$ is not arc-transitive. 
Then $\bcay(G,S)$ is a BCI-graph. 
\end{thm}

Given a group $G$ with identity element $1_G,$ we shall use the symbols 
$\nn$ and $\ee$ to denote the elements $(1_G,0)$ and $(1_G,1),$ respectively, from $G \times \{0,1\};$ 
and for a subset $S \subseteq G,$ we write $(S,0) = \{ (s,0) : s \in S \}$ and $(S,1) = \{ (s,1) : s \in S \}$. 
For $g \in G,$ let $\hat{g}$ be the permutation of $G \times \{0,1\}$ defined by 
$(x,i)^{\hat{g}} = (x g,i)$ for every  $x \in G$ and $i \in \{0,1\},$ and let 
$\hat{G} = \{ \hat{g} : g \in G\}$.  Obviously, $\hat{G} \le \aut(\bcay(G,S))$ always holds.

\begin{lem}\label{tauX}
Let $\Gamma$ be a cubic bipartite graph with bipartition classes $\Delta_i,$ $i=1,2,$ and $X \le \aut(\Gamma)$ be a semiregular 
subgroup whose orbits are $\Delta_i,$ $i=1,2,$ and $X \in \Cs$. 
Then $\aut(\Gamma)$ has an element $\tau_X$ which satisfies:
\begin{enumerate}[(i)]
\item  every subgroup of $X$ is normal in $\langle X, \tau_X \rangle;$ 
\item $\langle X, \tau_X \rangle \le \aut(\Gamma)$ is regular on $V(\Gamma)$. 
\end{enumerate}
\end{lem}

\begin{proof}
It is straightforward to show that $\Gamma \cong \bcay(X,S)$ for some subset $S \subseteq X$ with $1_X \in S$ and $|S|=3$.  
Moreover, there is an isomorphism from $\Gamma$ to $\bcay(G,S)$ which induces a permutation isomorphism from 
$X$ to $\hat{X} \le \aut(\bcay(G,S))$. Therefore, it is sufficient to find  $\tau \in \aut(\bcay(X,S))$ for which
every subgroup of $\hat{X}$  is normal in $\langle \hat{X}, \tau \rangle;$ and 
$\langle \hat{X}, \tau \rangle \le \aut(\bcay(X,S))$  is regular on $V(\bcay(X,S))$. 

Since $X \in \Cs,$ $X = U \times V,$ where $U$ is 
an Abelian group of odd order, and $V$ is trivial or one of $Z_{2^r}, \Z_2^r$ and $\Q_8$.
We prove below the existence of an automorphism $\iota \in \aut(X),$ which maps the set  $S$ to its inverse 
$S^{-1} = \{x^{-1} : x \in S\}$. 
Let $\pi_U$ and $\pi_V$ denote the projections $U \times V \to U$ and $U \times V \to V$ respectively. 
It is sufficient to find an automorphism 
$\iota_1 \in \aut(U)$ which maps $\pi_U(S)$ to $\pi_U(S)^{-1},$ and an   
automorphism $\iota_2 \in \aut(V)$ which maps $\pi_V(S)$ to $\pi_V(S)^{-1}$. 
Since $U$ is Abelian, we are done by choosing $\iota_1$ to be the automorphism $x \mapsto x^{-1}$. 
If $V$ is Abelian, then let $\iota_2 : x \mapsto x^{-1}$. Otherwise,  
$V \cong \Q_8,$ and since $|\pi_V(S) \setminus \{1_V\}| \le 2,$ it follows that $\pi_V(S)$ is conjugate to $\pi_V(S)^{-1}$ in $V$. 
This ensures that $\iota_2$ can be chosen to be 
some  inner automorphism.  Now, define $\iota$ by setting  its restriction $\iota|_U$ to $U$ as 
$\iota|_U = \iota_1,$ and  its restriction $\iota|_V$  to $V$ as $\iota|_V = \iota_2$.
Define the permutation $\tau$ of $X \times \{0,1\}$ by 
$$
(x,i)^\tau = \begin{cases} (x^\iota,1) & \mbox{ if } i=0, \\ 
(x^\iota,0) & \mbox{ if } i=1.  \end{cases}
$$
The vertex $(x,0)$ of $\bcay(X,S)$ has neighborhood $(Sx,1)$. This is mapped  by $\tau$ to the set $(S^{-1}x^\iota,0),$
which is equal to the neighborhood of $(x^\iota,1)$. We have proved that $\tau \in \aut(\bcay(X,S))$. 

It follows from its construction that $\tau$ is an involution. Fix an arbitrary subgroup $Y \le X,$ and pick $y \in Y$. 
We may write $y = y_U y_V$ for some $y_U \in U$ and $y_V \in V$. 
Then $\langle y_U, y_V \rangle \le Y,$ since $y_U$ and $y_V$ commute and $\gcd(|U|,|V|)=1$. 
Also,  $(y_U)^{\iota_1} = y_U^{-1}$ and $(y_V)^{\iota_2} \in \langle y_V \rangle$, implying that 
$y^\iota = (y_U)^{\iota_1} (y_V)^{\iota_2} \in \langle y_U,u_V \rangle \le Y$. 
We conclude that $\iota$ maps $Y$ to itself. 
Thus $\tau^{-1} \hat{y} \tau = \tau \hat{y} \tau = \hat{y^\iota}$ is in $\hat{Y},$ and $\tau$ normalizes $\hat{Y}$. 
Since $X \in \Cs,$  $\hat{Y}$ is also normal in $\hat{X},$ and  part (i) follows.

For part (ii), observe that $|\langle \hat{X},\tau \rangle| = 2|X| = |V(\bcay(X,S))|$. 
Clearly, $\langle \hat{X},\tau \rangle$ is transitive on $V(\bcay(X,S)),$ so it is regular. 
\end{proof}

The following result about Cayley digraphs is a special case of \cite[Lemma 3.1]{Bab77}:

\begin{lem}\label{B}
The following are equivalent for every Cayley digraph $\cay(G,S)$.
\begin{enumerate}[(i)]
\item $\cay(G,S)$ is a CI-graph.
\item Every two regular subgroups of $\aut(\cay(G,S)),$ isomorphic to $G,$ are 
conjugate  in $\aut(\cay(G,S))$.
\end{enumerate}
\end{lem}

We prove next an analog of the previous lemma for cubic bi-Cayley graphs on groups $G \in \Cs$. 
For a permutation group $H \le \sym(G \times \{1,0\}),$ we denote by $\G(H)$ the set of all semiregular subgroups of $H$ whose orbits are $(G,0)$ and $(G,1)$. 

\begin{lem}\label{B-type}
The following are equivalent for every bi-Cayley graph $\Gamma = \bcay(G,S),$ where $G \in \Cs$ and $|S|=3$.
\begin{enumerate}[(i)]
\item $\bcay(G,S)$ is a BCI-graph. 
\item  Every two subgroups in $\G(\aut(\Gamma)),$ 
isomorphic to $G,$ are conjugate in $\aut(\Gamma)$.
\end{enumerate}
\end{lem}

\begin{proof} 
We start with the part $(i) \Rightarrow (ii)$. 
Let $X \in \G(\aut(\Gamma))$ such that $X \cong G$. We have to show that 
$X$ and $\hat{G}$ are conjugate in $\aut(\Gamma)$. Let $i \in \{0,1\},$ and set $X^{(G,i)}$ and $\hat{G}^{(G,i)}$ 
for the permutation groups of the set $(G,i)$ induced by $X$ and $\hat{G}$ respectively. 
The groups $X^{(G,i)}$ and $\hat{G}^{(G,i)}$  are conjugate in $\sym((G,i)),$ because these are isomorphic and regular 
on $(G,i)$.  Thus $X$ and $\hat{G}$ are conjugate by 
a permutation $\phi \in \sym(G \times \{0,1\})$ such that $(G,0)$ is $\phi$-invariant, and we write $X = \phi \hat{G} \phi^{-1}$. 
Consider the graph $\Gamma^\phi,$ the image of $\Gamma$ under $\phi$. Then 
$\hat{G} = \phi^{-1} X \phi \le \aut(\Gamma^\phi)$. Using this and that $(G,0)$ is $\phi$-invariant, we obtain that 
$\Gamma^\phi = \bcay(G,T)$ for some subset $T \subseteq G$. Then $\Gamma \cong \bcay(G,T),$ and by (i), 
 $T = g S^\alpha$ for some $g \in G$ and $\alpha \in \aut(G)$. Define the permutation $\sigma$ of $G \times \{0,1\}$ by 
$$
(x,i)^\sigma = \begin{cases} (x^\alpha,0) & \text{ if } i=0, \\ (g x^\alpha,1) & \text{ if } i=1.  \end{cases}
$$
Notice that, $\sigma$ normalizes $\hat{G}$. 
The vertex $(x,0)$ of $\bcay(X,S)$ has neighborhood $(Sx,1)$. These are mapped  by $\sigma$ to the 
vertex $(x^\alpha,0)$ and the set $(g S^\alpha x^\alpha,1) = (T x^\alpha,1)$. This proves that 
$\sigma$ induces an isomorphism from $\Gamma$ to $\Gamma^\phi,$ and it follows in turn that, 
$\Gamma^\phi = \Gamma^\sigma,$ $\phi \sigma^{-1} \in \aut(\Gamma),$ and thus 
$\phi = \rho \sigma$ for some $\rho \in \aut(\Gamma)$. Finally, 
$X = \phi \hat{G} \phi^{-1} = \rho \sigma \hat{G} \sigma^{-1} \rho^{-1} = \rho \hat{G} \rho^{-1},$
i.e., $X$ and $\hat{G}$ are conjugate in $\aut(\Gamma)$. \medskip

We turn to the part $(i) \Leftarrow (ii)$. 
Let $\Gamma' = \bcay(G,T)$ such that $\Gamma' \cong \Gamma$. We have to show that $T = g S^\alpha$ for some 
$g \in G$ and $\alpha \in \aut(G)$. 
We claim the existence of an isomorphism $\phi : \Gamma \to \Gamma'$ for which 
$\phi : \nn \mapsto \nn$ and $(G,0)$ is $\phi$-invariant (here $\phi$ is viewed as a permutation of $G \times \{0,1\}$). 
We construct $\phi$ in a few steps. To start with, choose an arbitrary isomorphism $\phi_1 : \Gamma \to \Gamma'$. 
Let $\tau_{\hat{G}}$ be the automorphism of $\Gamma'$ defined in Lemma \ref{tauX}. 
Since $\langle \hat{G},\tau_{\hat{G}} \rangle$ is regular on $V(\Gamma'),$   
there exists $\rho \in \langle \hat{G},\tau_{\hat{G}} \rangle$ which maps  $\nn^{\phi_1}$ to $\nn$.
Let $\phi_2 = \phi_1 \rho$. Then $\phi_2$ is an isomorphism from 
$\Gamma$ to $\Gamma',$ and also $\phi_2 : \nn \mapsto \nn$.  The connected component of 
$\Gamma$ containing the vertex $(x,0)$ is equal to the induced subgraph $\Gamma[(xH,0) \cup (sx H,1)],$ where 
$s \in S$ and $H \le G$ is generated by the set $s^{-1} S$. 
It can be easily checked that 
$$
\Gamma[(x H,0) \cup (s x H,1)] \cong \bcay(H,s^{-1} S).
$$ 
Similarly, 
the connected component of 
$\Gamma'$ containing the vertex $(x,0)$ is equal to the induced subgraph $\Gamma'[(x K,0) \cup (t x K,1)],$
where $t \in T$ and $K \le G$ is generated by the set $t^{-1} T,$ and 
$$
\Gamma'[(xK,0) \cup (tx K,1)] \cong \bcay(K,t^{-1} T).
$$ 
Since $\phi_2$ fixes $\nn,$ it induces an isomorphism from 
$\Gamma[(H,0) \cup (sH,1)]$ to $\Gamma[(K,0) \cup (t K,1)];$ denote this isomorphism by $\phi_3$. 
It follows from the connectedness of these induced subgraphs that 
$\phi_3$ preserves their bipartition classes, moreover, $\phi_3$ maps $(H,0)$  to $(K,0),$ since it fixes $\nn$.  
Finally, take $\phi : \Gamma \to \Gamma'$ to be the isomorphism whose restriction to each component of $\Gamma$ equals 
$\phi_3$. It is clear that $\phi : \nn \mapsto \nn$ and $(G,0)$ is $\phi$-invariant. 

Since $\hat{G} \le \Gamma',$ $\phi \hat{G} \phi^{-1} \le \aut(\Gamma)$. The orbit of $\nn$ under $\phi \hat{G} \phi^{-1}$ 
is equal to $(G,0)^{\phi^{-1}} = (G,0),$ and hence $\phi \hat{G} \phi^{-1} \in \G(\aut(\Gamma))$. 
By (ii), $\phi \hat{G} \phi^{-1} = 
\sigma^{-1} \hat{G} \sigma$ for some $\sigma \in \aut(\Gamma)$. 
By Lemma \ref{tauX}, the normalizer of $\hat{G}$ in 
$\aut(\Gamma)$ is transitive, implying that $\sigma$ can be chosen so that $\sigma : \nn \mapsto \nn$. 
To sum up, we have an isomorphism $(\sigma \phi) : \Gamma \mapsto \Gamma'$ which fixes 
$\nn$ and  also normalizes $\hat{G}$. Thus $(\sigma \phi)$ maps $(G,1)$ to itself. 
Let $G_{\text{right}} \le \sym(G,1)$ be the permutation group induced by the action of $\hat{G}$ on $(G,1)$. 
Then the permutation of $(G,1)$ induced by $(\sigma \phi)$  belongs to the holomorph of $G_{\text{right}}$ 
(cf. \cite[Exercise 2.5.6]{DixM96}), and therefore, there exist $g \in G$ and $\alpha \in \aut(G)$ such that 
$(\sigma \phi) : (x,1) \mapsto (g x^\alpha,1)$ for all $x \in G$.
On the other hand, being an isomorphism from $\Gamma$ to $\Gamma',$  $\sigma \phi$ maps  
$(S,1)$  to $(T,1)$. These give that $(T,1) = (S,1)^{\sigma \phi} = (g S^\alpha,1),$ i.e., $T=g S^\alpha$. 
\end{proof}

In the following lemma we connect the BCI-property with the CI-property.  

\begin{lem}\label{B-type-cor}
Let $\Gamma = \bcay(G,S), G \in \Cs,$ $|S|=3,$ and suppose, in addition, that  
$\aut(\Gamma)_\nn = \aut(\Gamma)_\ee$. Then 
$\cay(G,S)$ is a CI-graph $\iff$ $\bcay(G,S)$  is a BCI-graph. 
\end{lem}

\begin{proof}
Set $A = \aut(\Gamma)$ and $A^+ = A_{\{ (G,0) \}}$. Obviously, $X \le A^+$ for every 
$X \in \G(A)$. Let $\tau_{\hat{G}} \in A$ be the automorphism defined in Lemma \ref{tauX}.  
Then $A = A^+ \rtimes \langle \tau_{\hat{G}} \rangle$. By Lemma \ref{tauX}.(i), 
$\tau_{\hat{G}}$ normalizes $\hat{G}$, hence the conjugacy class of subgroups of $A$ containing $\hat{G}$ is equal to 
the conjugacy class of subgroups of $A^+$ containing $\hat{G}$. 
Using this and Lemma \ref{B-type}, (ii) follows if every  $X \in \G(A),$ isomorphic to $G,$ 
is conjugate to $\hat{G}$ in $A^+$.  

Let $\Delta = \{\nn,\ee\}$ and consider the setwise stabilizer $A_{\{\Delta\}}$. 
Since $A_\nn = A_\ee,$ $A_\nn \le A_{\{\Delta\}}$.
By \cite[Theorem 1.5A]{DixM96},  the orbit of $\nn$ under  $A_{\{\Delta\}}$ is a block for $A$. 
Since $\tau_{\hat{G}}$ switches $\nn$ and $\ee,$ this orbit is equal to $\Delta,$ and the system of blocks induced by $\Delta$ is 
$$ 
\delta = \{\Delta^{\hat{x}} : x \in G \} = \big\{ \, \{ (x,0), (x,1) \} : x \in G \, \big\}.
$$ 
This allows us to define the action of $A$ on $G$ by letting $x^\sigma = x'$ if $\sigma$ maps the block $\{(x,0), (x,1)\} $ to the block $\{(x',0), (x',1) \}$. We write $\bar{\sigma}$ for the image of $\sigma$ under the corresponding  permutation 
representation, and let $\bar{X} = \{ \bar{\sigma} : \sigma \in X \}$ for a subgroup $X \le A$. 
Notice that, the subgroup $A^+ < A$ is faithful in this action. In particular, two subgroups 
$X$ and $Y$ of $A^+$ are conjugate in $A^+$ exactly when $\bar{X}$ and $\bar{Y}$ are conjugate in $\bar{A^+}$. 
Also, for every $X \le A^+,$
$X  \in \G(A)$  and $X  \cong G$ if and only if $\bar{X}$ is regular on $G$ and $\bar{X} \cong G$.

We prove next that $\bar{A^+} = \aut(\cay(G,S))$. 
Pick an automorphism $\sigma \in A^+$ and an arc $(x,sx)$ of $\cay(G,S)$.
The edge $\{ (x,0),(sx,1)\}$ of $\Gamma$ is mapped by $\sigma$ to an edge $\{ (x',0),(s'x',1) \}$ for some 
$x' \in G$ and $s' \in S,$ hence $\bar{\sigma} :  x \mapsto x'$ and $s x \mapsto s' x',$ i.e.,  it 
maps the arc $(x,sx)$ to the arc $(x',s'x')$. We have just proved 
that $\bar{\sigma} \in \aut(\cay(G,S)),$ and hence $\bar{A^+} \le \aut(\cay(G,S))$.
In order to establish the relation ``$\ge$", for an arbitrary automorphism $\rho \in \aut(\cay(G,S)),$ define the permutation 
$\pi$ of $G \times \{0,1\}$ by 
$(x,i)^\pi = (x^\rho,i)$ for all $x \in G$ and $i \in \{0,1\}$. 
It is easily checked that $\pi \in A^+$ and $\bar{\pi} = \rho$. Thus $\bar{A^+} \ge \aut(\cay(G,S)),$ and so $\bar{A^+} = \aut(\cay(G,S))$. 

Now, the desired  equivalence follows along the following lines:
\begin{eqnarray*}
(i) & \stackrel{(a)}{\iff} & \text{every two regular subgroup of $\bar{A^+},$ isomorphic to $G,$ are conjugate in $\bar{A^+}$ } \\
     & \stackrel{(b)}{\iff} &  \text{every two subgroups in $\G(A),$ isomorphic to $G,$  are conjugate in $G$ } \\
     & \stackrel{(c)}{\iff}& (ii),
\end{eqnarray*}   
where (a) is Lemma \ref{B}, (b) is proved above, and (c) is Lemma \ref{B-type}.  
\end{proof}

Now it is easy to prove Theorem \ref{Main-non-AT}. \smallskip

\begin{proof}[Proof of Theorem \ref{Main-non-AT}]
Since $\Gamma$ is vertex-transitive (see Lemma \ref{tauX}), but not arc-transi- 
tive, we have 
$A_\nn = A_{(s,1)}$ for some $s \in S$.
We show below that $\bcay(G,s^{-1} S)$ is a BCI-graph, this obviously yields that the same holds for $\bcay(G,S)$. 
Define the permutation $\phi$ of $G \times \{0,1\}$ by 
$$
(x,i)^\phi = \begin{cases} (x,0) & \mbox{ if } i=0, \\ 
(s^{-1} x,1) & \mbox{ if } i=1.  \end{cases}
$$
We showed before that $\phi$ induces an isomorphism from $\Gamma$ to $\Gamma' = \bcay(G,s^{-1} S)$. 
Then we have $ \aut(\Gamma')_\nn = \phi^{-1} A_\nn \phi =  \phi^{-1} A_{(1,s)} \phi = \aut(\Gamma')_\ee$.
Thus Lemma \ref{B-type-cor} applies to $\Gamma',$ as a result, it is sufficient to show that 
$\cay(G,s^{-1} S)$ is a CI-graph. This follows because $|s^{-1} S \setminus \{1_G\}|=2$ and that $G$ is a 
$2$-DCI-group (see \cite[Theorem 1.3]{Li99}).  
\end{proof}

\section{Proof of Theorem \ref{Main}}

Let $\Gamma$ be an arbitrary finite graph and $G \le \aut(\Gamma)$ which is transitive on $V(\Gamma)$.
For a normal subgroup  $N \triangleleft G$ which is not transitive on $V(\Gamma),$ the \emph{quotient graph} $\Gamma_N$ 
is the graph whose vertices are the $N$-orbits on 
$V(\Gamma),$ and two $N$-orbits $\Delta_i, i=1,2,$ are adjacent if and only if there exist 
$v_i \in \Delta_i, i=1,2,$ which are adjacent in 
$\Gamma$. For a positive integer $s,$ an \emph{$s$-arc} of $\Gamma$ is an ordered $(s+1)$-tuple $(v_0,v_1,\ldots,v_s)$ of 
vertices of $\Gamma$ such that, for every $i \in \{1,\ldots,s\},$ $v_{i-1}$ is adjacent to $v_i,$ and for every 
$i \in \{1,\ldots,s-1\},$ $v_{i-1} \ne v_{i+1}$. The graph $\Gamma$ is called \emph{$(G,s)$-arc-transitive} 
(\emph{$(G,s)$-arc-regular}) if $G$ is transitive (regular) on 
the set of $s$-arcs of $\Gamma$. If $G = \aut(\Gamma),$ then a $(G,s)$-arc-transitive ($(G,s)$-arc-regular) graph is simply called 
\emph{$s$-transitive} (\emph{$s$-regular}). The proof of the following lemma is straightforward, hence it is omitted (it can 
be also deduced from \cite[Theorem 9]{Lor84}).

\begin{lem}\label{Gamma_N}
Let $\Gamma = \bcay(G,S)$ be a connected arc-transitive graph,  
$G$ be any finite group, $|S|=3,$ and $N < \hat{G}$  be a subgroup which is normal in $\aut(\Gamma)$. Then the following hold:
\begin{enumerate}[(i)]
\item  $\Gamma_N$ is a cubic connected arc-transitive graph.
\item $N$ is equal to the kernel of $\aut(\Gamma)$ acting on the set of $N$-orbits.
\item $\Gamma_N$ is isomorphic to a bi-Cayley graph of the group $\hat{G}/N$.
\end{enumerate} 
\end{lem}

\begin{rem}\label{Gamma_N-rem}
Let $\Gamma$ and $N$ be as described in Lemma \ref{Gamma_N}.
The group $\aut(\Gamma)$ acts on the set of $N$-orbits, i.e., on the vertex set $V(\Gamma_N)$. 
Lemma \ref{Gamma_N}.(ii) implies that, the  induced permutation group on $V(\Gamma_N)$ is isomorphic to 
$\aut(\Gamma)/N,$ and therefore, by some abuse 
of notation, this permutation group will also be denoted by $\aut(\Gamma)/N$.  
In what follows we shall write $\aut(\Gamma)/N \le \aut(\Gamma_N)$. Also note that,
 if $\Gamma$ is $s$-transitive, then 
$\Gamma_N$ is $(\aut(\Gamma)/N,s)$-arc-transitive. 
\end{rem}

The proof of Theorem \ref{Main} will be based on three lemmas about cubic connected arc-transitive 
bi-Cayley graphs, to be proved  below.  In order to simplify the formulations, we keep 
the following notation in all lemmas: 
\begin{enumerate}
\item[$(\ast)$] \ 
$\Gamma = \bcay(G,S)$  is a connected arc-transitive graph, where $G \in \Cs$ and $|S|=3$. 
\end{enumerate}

\begin{lem}\label{AT1}
With notation $(\ast),$ let $\delta$ be a system of blocks for $\aut(\Gamma)$ induced by a block properly contained in $(G,0),$
and $X$ be in $\G(\aut(\Gamma))$ such that $X \in \Cs$.  Then $A_\delta <  X$.
\end{lem}

\begin{proof} 
Set $A = \aut(\Gamma)$. Let $Y = X \cap A_{\{ \Delta \}},$ where $\Delta \in \delta$ with $\Delta \subset (G,0)$.  
Then $\Delta$ is equal to an orbit of $Y,$ and $|Y|=|\Delta|$ because $\Delta \subset (G,0)$ and 
$X$ is regular on $(G,0)$. Formally, $\Delta = \orb_Y(v)$ for some vertex $v \in \Delta$. 

Let $\tau_X \in A$ be the automorphism defined in Lemma \ref{tauX}, and set  $L = \langle X,\tau_X \rangle$. The group 
$L$  is regular on $V(\Gamma),$ and $Y \trianglelefteq L$.   These yield
$$
\delta = \{  \Delta^l : l \in L \} = \{ \orb_Y(v)^l : l \in L \} = \{ \orb_Y(v^l) : l \in L \}.
$$ 
From this $Y \le A_\delta$. Since $\delta$ has more than $2$ blocks, and $\Gamma$ is a connected and cubic graph, it  is known 
that $A_\delta$ is semiregular. These imply that $A_\delta = Y < X$. 
\end{proof}
 
\begin{cor}\label{AT1-cor}
With notation $(\ast),$ let $N < \hat{G}$ be normal in $\aut(\Gamma),$ and 
$X$ be in $\G(\aut(\Gamma))$ such that $X \in \Cs$. Then $N < X$.   
\end{cor}

\begin{proof} 
Let $\delta$ be the system of blocks for $\aut(\Gamma)$ consisting of the $N$-orbits. 
Then $A_\delta = N$ by Lemma \ref{Gamma_N}.(ii), and the corollary follows directly from Lemma \ref{AT1}. 
\end{proof}  

We denote by $Q_3$ the graph of the cube and by $\H$  the Heawood graph. 
Recall that, the \emph{core} of a subgroup $H \le K$ in the group $K$ is the largest normal subgroup of $K$ contained in $H$. 

\begin{lem}\label{AT2}
With notation $(\ast),$ suppose that $\hat{G}$ is not normal in $\aut(\Gamma),$ and let 
$N$ be the core of $\hat{G}$ in $\aut(\Gamma)$.
Then  $(\hat{G}/N,\Gamma_N)$ is isomorphic to one of the pairs 
$(\Z_3,K_{3,3}),$ $(\Z_4,Q_3),$ and $(\Z_7,\H)$. 
\end{lem}

\begin{proof}
Set $A=\aut(\Gamma)$. Consider the quotient graph $\Gamma_N,$ and let  
$M \le \hat{G}$ such that $N \le M$ and $M/N \trianglelefteq \aut(\Gamma_N)$ (here $M/N \le A/N \le \aut(\Gamma_N),$ 
see Remark \ref{Gamma_N-rem}).
This implies in turn that, $M/N \trianglelefteq  A/N,$ $M \trianglelefteq A,$ and $M = N$. 
We conclude that, $\Gamma_N$ is a bi-Cayley graph of $\hat{G}/N,$ $\hat{G}/N$ is in $\Cs,$ and 
$\hat{G}/N$ has trivial core in $\aut(\Gamma_N)$. This shows that it is sufficient to prove Lemma \ref{AT2} in the particular case 
when $N$ is trivial. For the rest of the proof we assume that the core $N$ is trivial, and we write $N=1$.  

By Tutte's theorem \cite{Tut47}, $\Gamma$ is $k$-regular for some $k \le 5$. 
Set $A^+ = \aut(\Gamma)_{\{(G,0)\}}$. Then $A = \langle A^+, \tau_{\hat{G}} \rangle,$ where 
$\tau_{\hat{G}} \in A$ is the automorphism defined in Lemma \ref{tauX}.  Let $M$ be the core of $\hat{G}$ in 
$A^+$. Then $M \trianglelefteq A,$ since $M$ is normalized by $\tau_{\hat{G}}$ by Lemma \ref{tauX}.(i), and 
$A = \langle A^+, \tau_{\hat{G}} \rangle$. Thus $M \le N =1,$  hence $M$ is also trivial. 
Consider $A^+$ acting on the set $[A^+:\hat{G}]$ of right $\hat{G}$-cosets in $A^+$. 
This action is faithful because $M$ is trivial. Equivalently, $G$ is embedded into 
$S_{3 \cdot 2^{k-1}-1},$ and we will write below that $G \le S_{3 \cdot 2^{k-1}-1}$.
Also, $A_\nn$ is determined uniquely by $k,$ and 
we have, respectively, $A_\nn \cong  \Z_3,$ or $S_3,$ or $D_{12},$ or $S_4,$ or $S_4 \times \Z_{12}$. 
We go through each case. \medskip

\noindent  CASE 1. $k=1$. \smallskip

This case can be excluded at once by observing that we have $G \le S_2$ by the above discussion,  
which  contradicts the obvious bound $|G| \ge 3$. \medskip 
 
\noindent CASE 2. $k=2$.  \smallskip 

In this case $G \le S_5$. Using also that $G \in \Cs,$ we see that $G$ is Abelian, hence $|G| \le 6,$ $|V(\Gamma)| \le 12$. 
We obtain by \cite[Table]{ConD02} that $\Gamma \cong Q_3,$ and $G \cong \Z_4$. \medskip

\noindent CASE 3. $k=3$.  \smallskip  

Then $A^+ = \hat{G} A_\nn = \hat{G} D_{12},$ a product of a nilpotent and a dihedral subgroup. Thus $A^+$ is solvable by 
Huppert-It\^{o}'s theorem (cf. \cite[13.10.1]{Sco64}). Since the core $N = 1,$ $A^+$ is primitive on $(G,0),$ see Lemma \ref{AT1}. Therefore, $G$ is a $p$-group. We see that $G$ is either Abelian or it is $\Q_8$. 
In the latter case $|V(\Gamma)|=16,$ and  
$\Gamma$ is isomorphic to the Moebius-Kantor graph, which is, however, $2$-regular (see \cite[Table]{ConD02}).
Therefore, $G$ is an Abelian $p$-group. Let $S = \{s_1,s_2,s_3\}$. Since $G$ is Abelian, for $\Gamma$ 
we have: 
$$
\nn \sim  (s_1,1)  \sim (s_2^{-1} s_1,0)  \sim (s_3 s_2^{-1} s_1,1) = (s_1 s_2^{-1} s_3,1) \sim (s_2^{-1} s_3,0) \sim 
(s_3,1) \sim \nn.
$$
Thus $\Gamma$ is of girth at most $6$. 
It was proved in \cite[Theorem 2.3]{ConN07} that the Pappus graph on $18$ points and the 
Deargues graph on $20$ points are the only 
$3$-regular cubic graphs of girth $6$. 
For the latter graph $|G|=10,$ contradicting that $G$ is a $p$-group. 
We exclude the former graph by the help of the computer package \texttt{Magma} \cite{BosCP97}.
We compute that the Pappus graph has no Abelian semiregular automorphism group  of order $9$ which 
has trivial core in the full automorphism group. Thus $\Gamma$  is of girth $4$ ($3$ and $5$ are impossible as the graph is bipartite). 
It is well-known that there are only two cubic arc-transitive graphs of girth $4$ (see also \cite[page 163]{KutM09}):  
$K_{3,3}$ and $Q_8$.  We get at once that  $\Gamma \cong K_{3,3}$ and $G \cong \Z_3$. \medskip

\noindent CASE 4. $k=4$.  \smallskip  

It is sufficient to show that $G$ is Abelian. 
Then by the above reasoning $\Gamma$ is of girth $6,$ and as the Heawood graph 
is the only cubic $4$-regular graph of girth $6$ (see \cite[Theorem 2.3]{ConN07}),  we get at once 
that  $\Gamma \cong \H$ and $G \cong \Z_7$.

Assume, towards a contradiction, that $G$ is non-Abelian. Thus $G = U \times V,$  where $U$ is an 
Abelian group of odd order, and $V \cong \Q_8$.  Since the core $N = 1,$ $A^+$ is primitive on $(G,0),$ 
see Lemma \ref{AT1}. In other words, $\Gamma$ is a $4$-transitive bi-primitive cubic graph. 
Two possibilities can be deduced from the list of $4$-transitive bi-primitive graphs given in \cite[Theorem 1.4]{Li01}:
\begin{itemize}
\item  $\Gamma$ is the standard double cover of a 
connected vertex-primitive cubic $4$-regular graph, in which case  $A  = 
A^+ \times \langle \eta \rangle$ for  an involution $\eta;$ or  
\item $\Gamma$ isomorphic to the sextet graph $S(p)$ (see \cite{BigH83}), where $p \equiv \pm 7(\text{mod }16),$  
in which case $A \cong PGL(2,p),$ and $A^+ \cong PSL(2,p)$.
\end{itemize}

The second possibility cannot occur, because then $A^+ \cong PSL(2,p),$ whose Sylow $2$-subgroup is a dihedral 
group (cf. \cite[Satz 8.10]{Hup67}), which contradicts that $V \le \hat{G} \le A^+,$ and $V \cong  \Q_8$.
It remains to exclude the first possibility. 
We may assume, by replacing $S$ with $x S$ for a suitable $x \in G$ if necessary, that 
$\eta$ switches $\nn$ and $\ee$. 
Since $\eta$ commutes with $\hat{G},$ we find $(x,0)^\eta = \nn^{\hat{x} \eta} = \nn^{\eta \hat{x}} = \ee^{\hat{x}}=(x,1)$ for every $x \in G$.  Let $s \in S$. Then $\nn \sim (s,1),$ hence $\ee = \nn^\eta \sim (s,0)^\eta = (s,1),$ 
which shows that $s \in S^{-1},$ and thus $S = S^{-1}$. Thus there exists $s \in S$ with $o(s) \le 2$. 
Put $T = s^{-1} S = s S$. Then $1_G \in T,$ and since $\Gamma$ is connected, $G = \langle T \rangle$. 
Notice that $s \in Z(G)$. This implies that $T^{-1} = T,$ and thus $\pi_V(T)$ satisfies 
$1_V \in \pi_V(T)$ and $\pi_V(T) = \pi_V(T)^{-1}$. Since $V \cong \Q_8,$ this implies that $\langle \pi_V(T) \rangle \ne V,$ 
a contradiction to $G = \langle T \rangle$. This completes the proof of this case. \medskip

\noindent CASE 5. $k=5$. \smallskip

In this case $\Gamma$ is a $5$-transitive bi-primitive cubic graph. 
It was proved in \cite[Corollary 1.5]{Li01} that 
$\Gamma$ is isomorphic to either the 
$P\Gamma L(2,9)$-graph on $30$ points (also known as the Tutte's 8-Cage), or the standards double 
cover of the $PSL(3,3).\Z_2$-graph on $468$ points. These graphs are of girth $8$ and $12$ respectively 
(see \cite[Table]{ConD02}). Also, in both cases $8 \nmid |G|,$ hence $G$ is Abelian. By this, however, 
$\Gamma$ cannot be of girth larger than $6$. This proves that this case does not  occur. 
\end{proof}

For a group $A$ and a prime $p$ dividing $|A|,$ we let $A_p$  denote a Sylow $p$-subgroup of $A$. 

\begin{lem}\label{AT3}
With notation $(\ast),$ let $X \in \G(\aut(\Gamma))$ such that $X \in \Cs$ and $X_2 \cong G_2$.
Then $X$ and $\hat{G}$ are conjugate in $\aut(\Gamma)$.   
\end{lem}

\begin{rem}
We remark that, the assumption $X_2 \cong G_2$ cannot be deleted. The Moebius-Kantor graph 
is a bi-Cayley graph of the group $\Q_8,$ which has a semiregular cyclic group of automorphism of order $8$ 
which preserves the bipartition classes. 
\end{rem}

\begin{proof}  
Set $A=\aut(\Gamma)$. The proof is split  into two parts according to whether $\hat{G}$ is normal in $A$. 
\medskip

\noindent  CASE 1. $\hat{G}$ is not normal in $A$.  \smallskip  

Let $N$ be the core of $\hat{G}$ in $A$. By Corollary \ref{AT1-cor}, 
 $N < X \cap \hat{G}$. Therefore, it  is sufficient to show that 
\begin{equation}
\text{ $X/N$ and $\hat{G}/N$ are conjugate in $A/N$}. 
\end{equation}
Recall that, the group $A/N \le \aut(\Gamma_N)$ for the quotient graph $\Gamma_N$ induced by $N$ 
(see Remark  \ref{Gamma_N-rem} and the preceding paragraph). Both groups $X/N$ and $\hat{G}/N$ are 
semiregular whose orbits are the bipartition classes of $\Gamma_N$. 
Also notice that, $\hat{G}/N$ cannot be normal in $A/N,$ otherwise $\hat{G}$ were normal in $A$.

According to Lemma \ref{AT2},  $(\hat{G}/N,\Gamma_N) \cong (\Z_3,K_{3,3}),$ or $(\Z_4,Q_8),$ or $(\Z_7,\H)$. 
Thus (1) follows immediately from Sylow's theorem when $(\hat{G}/N,\Gamma_N) \cong (\Z_7,\H)$. 

Let $(\hat{G}/N,\Gamma_N) \cong (\Z_3,K_{3,3})$. Since $\hat{G}/N$ is not normal in $A/N,$ and $\Gamma_N$ is 
$(A/N,1)$-arc-transitive, we compute by \texttt{Magma} that $A/N = \aut(\Gamma_N),$ or it is a unique subgroup of
$\aut(\Gamma_N)$ of index $2$. In both cases $A/N$ has one conjugacy classes of semiregular subgroups whose orbits 
are the bipartition classes of $\Gamma_N$. Thus (1) holds. 

Let $(\hat{G}/N,\Gamma_N) \cong (\Z_4,Q_8)$. Since $X_2 \cong G_2,$ $N/X \cong \hat{G}/N \cong \Z_4$. 
Using this and that $\Gamma_N$ is $(A/N,1)$-arc-transitive, we compute by \texttt{Magma} that $A/N = \aut(\Gamma_N),$ 
and that $\aut(\Gamma_N)$ has one conjugacy classes of semiregular cyclic subgroups whose orbits 
are the bipartition classes of $\Gamma_N$. Thus (1) holds also in this cases. 
\medskip

\noindent  CASE 2. $\hat{G}$ is normal in $A$.  \smallskip  

We have to show that $X = \hat{G}$. Notice that, $X$ contains every proper subgroup 
$K < \hat{G}$ which is characteristic in $\hat{G}$. Indeed, since $\hat{G} \trianglelefteq A,$  $K \trianglelefteq A,$ 
and hence $K <  X$ follows by Corollary \ref{AT1-cor}. 
This property will be used often below. 

In particular, $\hat{G}_p \le G$ is characteristic for every prime $p$ dividing $|\hat{G}|$. 
Thus  $\hat{G}_p < X$ if $G$ is not a $p$-group, hence $X = \hat{G}$. 
Let $G$ be a $p$-group. If $p > 3,$ then both 
$\hat{G}$ and $X$ are Sylow $p$-subgroups of $A,$ and the statement follows by Sylow's theorem. 
Notice that, since $\Gamma$ is connected, $G$ is generated by the set $s^{-1} S$ for $s \in S,$ hence it 
is generated by two elements.  

Let $p=2$. Assume for the moment that $G$ is cyclic. Then $\hat{G}$ has a characteristic subgroup $K$ 
such that $\hat{G}/K \cong \Z_4$. Then $K \trianglelefteq A,$ 
$\Gamma_K \cong Q_3,$  $\Gamma_K$  is a bi-Cayley graph of $\hat{G}/K,$ which is 
 is normal in $A/K \le \aut(\Gamma_K)$. A simple computation shows that this situation does not occur. 
Let $G$ be a non-cyclic group in $\Cs$. Therefore, $G \cong \Z_2^2$ and 
$\Gamma \cong Q_3,$ or $G \cong \Q_8$ and  
$\Gamma$ is the Moebius-Kantor graph.  Now, $X = X_2 \cong G_2 = G$. Using this,  
$X = \hat{G}$ can be verified by the help of \texttt{Magma} in either case. 

Let $p = 3$. Observe first that $|G| > 3$. For otherwise, $\Gamma \cong K_{3,3},$ but no semiregular 
automorphism group of order $3$ is normal in $\aut(K_{3,3})$. 
Assume for the moment that $G \not\cong \Z_{3^e} \times \Z_{3^e}$ for any $e \ge 1$. 
In this case $\hat{G}$ has a characteristic subgroup $K$ 
such that $\hat{G}/K \cong \Z_9,$ or $\Z_9 \times \Z_3$. 
It follows that $K \trianglelefteq A,$ and $\Gamma_K$ is the Pappus graph, or 
the  unique cubic arc-transitive graph on $54$ points (see \cite[Table]{ConD02}). 
Moreover, it is a bi-Cayley graph of $\Z_9$ in the first, and of  
$\Z_9 \times \Z_3$ in the second case. However, we have checked by \texttt{Magma} that none of these is possible; and 
therefore, $G \cong \Z_{3^e} \times \Z_{3^e}$ for some $e \ge 1$.
If $e=1,$ then $G \cong \Z_3^2,$ and $\Gamma$ is the Pappus graph. 
However, this graph has no automorphism group which  is isomorphic to $\Z_3^2$ and also 
normal in the full automorphism group. Therefore, $e > 1,$ and thus the subgroup $K = \langle  \hat{x}^9 : x \in G \rangle$ is characteristic in $\hat{G}$ of index $81$. 
It follows that $K <  X,$ and $\Gamma_K$ is the unique cubic arc-transitive graph on $162$ points (see \cite[Table]{ConD02}). 
A direct computation gives that $X/K = \hat{G}/K,$ which together with $K  <  X \cap \hat{G}$ yield 
that $X = \hat{G}$. 
\end{proof}

Recall that, a group $H$ is \emph{homogeneous} if every isomorphism between two subgroups of $H$ can be extended to 
an automorphism of $H$. The following result  is \cite[Proposition 3.2]{Li99}: 

\begin{prop}\label{L}  
Every $2$-DCI-group is homogeneous. 
\end{prop}

Since every group  in $\C$ is a $2$-DCI-group (see \cite[Theorem 1.3]{Li99}), we have the corollary that every group in $\C$ is homogeneous.
Everything is prepared to complete the proof of Theorem \ref{Main}. \smallskip

\begin{proof}[Proof of Theorem \ref{Main}] 
Let $G \in \C$ and $\Gamma = \bcay(G,S)$ such that $|S| \le 3$. We have to show that $\Gamma$ is a BCI-graph. 
This holds trivially when $|S|=1,$ and follows from the homogeneity of $G$ when $|S|=2$.
Let $|S|=3$. The claim is proved in Theorem \ref{Main-non-AT} when $\Gamma$ is not arc-transitive. 
For the rest of the proof we assume that $\Gamma$ is an arc-transitive graph.  \medskip

Let $\bcay(G,S) \cong \bcay(G,T)$ for some subset $T \subseteq G$. We may assume without loss of generality that 
$1_G \in S \cap T$. Let $H = \langle S \rangle$ and $K = \langle T \rangle$. 
Then $H,K \in \Cs,$ both bi-Cayley graphs $\bcay(H,S)$ and $\bcay(K,T)$ are connected, and 
$\bcay(H,S) \cong \bcay(K,T)$.
We claim that $\bcay(H,S)$ is a BCI-graph. 
In view of Lemma \ref{B-type}, this holds if for every 
$X \in \G(\aut(\bcay(H,S))),$ isomorphic to $H,$ $X$ and $\hat{H}$ are conjugate in  $\aut(\bcay(H,S))$. 
Now this follows directly from Lemma \ref{AT3}. 

Let $\phi$ be an isomorphism from $\bcay(K,T)$ to $\bcay(H,S),$ and consider the group $X = \phi^{-1} \hat{K} \phi \le 
\sym(H)$. Since $\phi$ maps the bipartition classes of $\bcay(K,T)$ to the bipartition classes of $\bcay(H,S),$ 
we have $X \in \G(\aut(\bcay(H,S)))$. Also, $X_2 \cong \hat{H}_2,$ because $X \cong K,$ $|H|=|K|$ and  
$H$ and $K$ are both contained  in the group $G$ from $\C$. Thus Lemma \ref{AT3} is applicable, as a result, $X$ and $\hat{H}$ 
are conjugate in $\aut(\bcay(H,S))$. In particular, $H \cong K,$ and since $G$ is homogeneous, there exists 
$\alpha_1 \in \aut(G)$ such that $K^{\alpha_1} =  H$. 
This $\alpha_1$ induces an isomorphism from $\bcay(K,T)$ to 
$\bcay(H,T^{\alpha_1})$. Therefore, $\bcay(H,S) \cong \bcay(H,T^{\alpha_1}),$ and since $\bcay(H,S)$ is a BCI-graph, 
$T^{\alpha_1} =   g S^{\alpha_2}$ for some $g \in H$ and $\alpha_2 \in \aut(H)$. 
By homogeneity, $\alpha_2$ extends to an automorphism of $G,$ implying eventually that $\bcay(G,S)$ is a BCI-graph. 
This completes the proof of the theorem. 
\end{proof}

\end{document}